\documentclass[12pt]{amsart}

\usepackage[margin=2.4cm]{geometry}
\usepackage{a4wide}

\usepackage{amsfonts,amsmath,amssymb,amsthm}
\usepackage{enumerate}

\newtheorem{theorem}{Theorem}[section]

\newtheorem{lemma}[theorem]{Lemma}
\newtheorem{proposition}[theorem]{Proposition}
\newtheorem{example}[theorem]{Example}

\newcommand*\xbar[1]{\hbox{\vbox{\hrule height 0.5pt \kern0.5ex\hbox{\kern-0.1em\ensuremath{#1}\kern-0.1em}}}}
\begin{document}

\author{Rafa{\l} Kapica, Janusz Morawiec}
\title{Inhomogeneous refinement equations with random affine maps}
\email[R. Kapica]{rkapica@math.us.edu.pl}
\email[J. Morawiec]{morawiec@math.us.edu.pl}
\address{Institute of Mathematics, University of Silesia, Bankowa 14, 40-007 Katowice, Poland}

\begin{abstract}
Given a probability space $(\Omega,{\mathcal A},P)$, random variables $L,M\colon\Omega\to\mathbb R$ and $g\in L^1(\mathbb R)$ we obtain two characterizations of these $f\in L^1(\mathbb R)$ which are solutions of the inhomogeneous refinement equation  with a random affine map of the form $f(x)=\int_\Omega |L(\omega)|f(L(\omega)x-M(\omega))P(d\omega)+g(x)$.
\end{abstract}


\maketitle

\section{Introduction}

Discrete inhomogeneous refinement equations of the form
$$
f(x)=\sum_{n\in\mathbb Z}c_nf(kx-n)+g(x)
$$ 
have been used in \cite{GHM1994} for a construction of  multiwavelets from a "fractal equation", in \cite{SN1996} for a construction of boundary wavelets and in \cite{SZ1998} for the study of convergence of cascade algorithms. Discrete inhomogeneous refinement equations and their continuous counterparts of the form
$$
f(x)=\int_{\mathbb R}f(kx-y)d\mu(y)+g(x)
$$ 
were studied in \cite{JJS2000} to provide a uniform and a complete characterization of the existence of their distributional solutions to cover all cases of interest. Some other motivations for the study of  discrete as well as continuous  inhomogeneous refinement equations can be found in  \cite{DH1998, DH1999, JJS2001,  JL2012, JLZ2009, L2000, L2001, L2003, L2004,  LS2010, LH1999, SZ2001, S2001,   Z2001}.
 
Motivation to write this paper comes from applications, because from that point of view all results on  the existence of "good" solutions of homogeneous as well as on inhomogeneous refinement equations are very important (see  \cite{C2000} where it is showed how nonexistence of "good" solutions of a refinement equation can lead to anomalous behavior of numerical methods for a construction of wavelets). Let us note that refinement equations always have plenty of "bad" solutions, even extremely strange (see e.g. \cite{BM2006, KM2013, M2001}). 

A common extension, both the above mentioned inhomogeneous refinement equations, is the following inhomogeneous refinement equation 
\begin{equation}\label{e}
f(x)=\int_{\Omega}|L(\omega)|f(L(\omega)x-M(\omega))P(d\omega)+g(x)
\end{equation}
with a random affine map $\varphi\colon \Omega\times\mathbb R\to\mathbb R$ defined by
\begin{equation}\label{rap}
\varphi(x,\omega)=L(\omega)x-M(\omega),
\end{equation}
where $L,M\colon\Omega\to\mathbb R$ are given random variables and $(\Omega,{\mathcal A},P)$ is a complete probability space. For more details on equation (\ref{e}) with $g=0$ we refer the reader to the survey paper \cite{KM2013} and the references therein.

First of all observe that if $g\colon\mathbb R\to\mathbb R$ is an integrable function and $f,h\colon\mathbb R\to\mathbb R$ are two representatives of a function from $L^1(\mathbb R)$ such that $(\ref{e})$ holds for almost all $x\in\mathbb R$, then $h$ also satisfies (\ref{e}) for almost all $x\in\mathbb R$; details are included in \cite{KM}. This observation allows us to accept the following definition: We say that $f\in L^1(\mathbb R)$ is an integrable solution of equation (\ref{e}), if (\ref{e}) holds for almost all $x\in\mathbb R$ (with respect to the one-dimensional Lebesgue measure).

Given $g\in L^1(\mathbb R)$ let $V_g$ denote the set of all integrable solutions of (\ref{e}). It is obvious that $V_g$ is a linear subspace of $L^1(\mathbb R)$ if and only if $g=0$. Moreover, if $g\neq 0$, then $V_g$ is uniquely determined by the subspace $V_0$ and by an arbitrary particular integrable solution of (\ref{e}), if exists any. It is also easy to see that equation (\ref{e}) has at most one integrable solutions (which is clearly nontrivial) if and only if $V_0=\{0\}$. 

In this paper we are interested in the set $V_g$. More precisely, we obtain two characterizations of these functions $f\in L^1(\mathbb R)$ which belong to the set $V_g$. Moreover, we also get results guaranteeing that the set $V_g$ is nonempty or consists of at most one member.

\section{Characterization of $V_g$ by Fourier transforms} 

From now on we fix a complete probability space $(\Omega,{\mathcal A},P)$, random variables $L,M\colon\Omega\to\mathbb R$, and a nontrivial function $g\in L^1(\mathbb R)$. 

If $P(L=0)>0$, then equation (\ref{e}) has exactly one integrable solution (see \cite[Corollary 3.4]{KM}, cf.  \cite[Remark 5.1]{KM2012}), and hence the set $V_g$ consists of exactly one member. In this paper it is assumed that 
$$
P(L=0)=0.
$$

Here and throughout, the Fourier transform of a function $f\in L^1(\mathbb R)$ is defined by 
$\widehat{f}(x)=\int_{\mathbb R}\exp\{itx\}f(t)dt$.

We begin with rewriting equation (\ref{e}) in the language of the Fourier transform. 

\begin{lemma}\label{lem1}
Assume $f\in L^1(\mathbb R)$. Then $f\in V_g$  if and only if 
\begin{equation}\label{2}
\widehat{f}(x)=\int_{\Omega}\exp\left\{ix\frac{M(\omega)}{L(\omega)}\right\}\widehat{f}\left(\frac{x}{L(\omega)}\right)P(d\omega)+\widehat{g}(x)
\end{equation}
for every $x\in\mathbb R$.
\end{lemma}

\begin{proof}
If $f\in V_g$, then using the Fourier transform to both sides of (\ref{e}) and applying the Fubini theorem we conclude that (\ref{2}) holds for every $x\in\mathbb R$.

On the other hand, if (\ref{2}) holds for every $x\in\mathbb R$, then applying the Fubini theorem we obtain
\begin{eqnarray*}
\int_{\mathbb R}\exp\{itx\}f(t)dt&=&
\int_{\Omega}\int_{\mathbb R}\exp\left\{ix\frac{M(\omega)+t}{L(\omega)}\right\}f(t)dtP(d\omega)
+\widehat{g}(x)\\
&=&\int_\Omega\int_{\mathbb R}\exp\{iux\}|L(\omega)|f(L(\omega)u-M(\omega))duP(d\omega)
+\widehat{g}(x)\\
&=&\int_{\mathbb R}\exp\{itx\}\left[\int_\Omega|L(\omega)|f(L(\omega)t-M(\omega))P(d\omega)
+g(t)\right]dt
\end{eqnarray*}
for almost all $x\in\mathbb R$. Since  $t \mapsto \int_{\Omega}|L(\omega)|f(L(\omega)t-M(\omega))P(d\omega)+g(t)$ is an $L^1$-function having the same Fourier transform as $f$, if follows that (\ref{e}) holds for almost all $x\in\mathbb R$.
\end{proof}

Putting $x=0$ in (\ref{2}) we get 
$$
\widehat{g}(0)=0,
$$
which is a necessary condition for the existence of integrable solutions of (\ref{e}).

Given a measurable function $\Psi\colon\Omega\to\mathbb R$ define a sequence $(\Psi_n)_{n\in\mathbb N}$ of random variables, acting on $\Omega^\infty$ with real values, putting
$$
\Psi_n(\varpi)=\Psi(\omega_n)
$$ 
for all $n\in\mathbb N$ and $\varpi=(\omega_1,\omega_2,...)\in\Omega^\infty$.
Next with every $h\in L^1(\mathbb R)$ we associate a sequence $(I_{h,n})_{n\in\mathbb N}$ of real functions defined by
$$
I_{h,n}(x)=\int_{\Omega^\infty}\exp\left\{ix\sum_{k=1}^n\frac{M_k(\varpi)}{L_1(\varpi)\cdots L_k(\varpi)}\right\}\widehat{h}\left(\frac{x}{L_1(\varpi)\cdots L_n(\varpi)}\right)P^\infty(d\varpi)
$$
for every $x\in\mathbb R$.

\begin{lemma}\label{lem2}
Assume that $f\in L^1(\mathbb R)$. Then $f\in V_g$  if and only if 
\begin{equation}\label{3}
\widehat{f}(x)=I_{f,N}(x)+\sum_{n=1}^{N-1}I_{g,n}(x)+\widehat{g}(x)
\end{equation}
for all $N\in\mathbb N$ and $x\in\mathbb R$. 
\end{lemma}

\begin{proof}
For $N=1$ formula (\ref{3}) coincides with formula (\ref{2}), if we adopt the convention that $\sum_{n=1}^{0}I_{g,n}(x)=0$ for every $x\in\mathbb R$. Hence Lemma \ref{lem1} implies that $f\in V_g$ if and only if for every $x\in\mathbb R$ formula (\ref{3}) holds with $N=1$. Therefore, to finish the proof it is enough to show that formula (\ref{3}) holds for all $N\geq 2$ and $x\in\mathbb R$ assuming that $f\in V_g$. 

Fix $N\in\mathbb N$, $f\in V_g$ and assume that (\ref{3}) holds for every $x\in\mathbb R$. Then by the induction hypothesis and Lemma \ref{lem1} we obtain
\begin{eqnarray*}
\widehat{f}(x)&=&\int_{\Omega^\infty}\exp\left\{ix\sum_{k=1}^N\frac{M_k(\varpi)}{L_1(\varpi)\cdots L_k(\varpi)}\right\}\widehat{f}\left(\frac{x}{L_1(\varpi)\cdots L_N(\varpi)}\right)P^\infty(d\varpi)\\
&&+\sum_{n=1}^{N-1}I_{g,n}(x)+\widehat{g}(x)\\
&=&\int_{\Omega^\infty}\exp\left\{ix\sum_{k=1}^N\frac{M_k(\varpi)}{L_1(\varpi)\cdots L_k(\varpi)}\right\}\left[\int_{\Omega}\exp\left\{ix\frac{M(\omega)}{L_1(\varpi)\cdots L_N(\varpi)L(\omega)}\right\}\right.\\
&&\hspace{3ex}\left.\widehat{f}\left(\frac{x}{L_1(\varpi)\cdots L_N(\varpi)L(\omega)}\right)P(d\omega)+\widehat{g}\left(\frac{x}{L_1(\varpi)\cdots L_N(\varpi)}\right)\right]P^\infty(d\varpi)\\
&&+\sum_{n=1}^{N-1}I_{g,n}(x)+\widehat{g}(x)\\
&=&I_{f,N+1}(x)+\sum_{n=1}^{N}I_{g,n}(x)+\widehat{g}(x)
\end{eqnarray*}
for every $x\in\mathbb R$. The proof is complete.
\end{proof}

We are now in a position to formulate our first result of this section.

\begin{theorem}\label{thm1}
Assume $f\in L^1(\mathbb R)$. If 
\begin{equation}\label{m}
P(M=0)<1
\end{equation}
and the series
\begin{equation}\label{4}
Z_\infty(\varpi)=\sum_{n=1}^\infty\frac{M_n(\varpi)}{L_1(\varpi)\cdots L_n(\varpi)}
\end{equation}
converges absolutely for almost all $\varpi\in\Omega^\infty$, then $f\in V_g$ if and only if
\begin{equation}\label{5}
\widehat{f}(x)=\widehat{f}(0)\int_{\Omega^\infty}\exp\{ixZ_\infty(\varpi)\}P^\infty(d\varpi)+
\sum_{n=1}^\infty I_{g,n}(x)+\widehat{g}(x)
\end{equation}
for every $x\in\mathbb R$.
\end{theorem}

\begin{proof}
According to Theorem 2.1 from \cite{GM2000}, we conclude that 
\begin{equation}\label{7b}
\lim_{n\to\infty}\frac{x}{L_1(\varpi)\cdots L_n(\varpi)}=0
\end{equation}
for all $x\in\mathbb R$ and almost all $\varpi\in\Omega^\infty$.

Assume $f\in V_g$. Then by the Lebesgue dominated convergence theorem we obtain
$$
\lim_{N\to\infty}I_{f,N}(x)=\int_{\Omega^\infty}\exp\{ixZ_\infty(\varpi)\}\widehat{f}(0)P^\infty(d\varpi)
$$
for every $x\in\mathbb R$. This jointly with Lemma \ref{lem2} shows that the series 
$\sum_{n=1}^\infty I_{g,n}(x)$ converges for every $x\in\mathbb R$ and, in consequence, (\ref{5}) holds for every $x\in\mathbb R$.

To prove the converse implication assume that (\ref{5}) holds for every  $x\in\mathbb R$. Then
\begin{eqnarray*}
&&\hspace{-20ex}\int_\Omega\exp\left\{ix\frac{M(\omega)}{L(\omega)}\right\}\widehat{f}\left(\frac{x}{L(\omega)}\right)P(d\omega)\\
\hspace{15ex}&=&\int_{\Omega}\exp\left\{ix\frac{M(\omega)}{L(\omega)}\right\}\left[\widehat{f}(0)\!\int_{\Omega^\infty}
\exp\left\{ix\frac{Z_\infty(\varpi)}{L(\omega)}\right\}P^\infty(d\varpi)\right.\\
&&\left.+\sum_{n=1}^\infty I_{g,n}\left(\frac{x}{L(\omega)}\right)+\widehat{g}
\left(\frac{x}{L(\omega)}\right)\right]P(d\omega)\\
&=&\widehat{f}(0)\int_{\Omega^\infty}\!\!\!\exp\left\{ix\frac{M_0(\varpi)}{L_0(\varpi)}
+ix\frac{Z_\infty(\varpi)}{L_0(\varpi)}\right\}P^\infty(d\varpi)\\
&&+\sum_{n=1}^\infty\int_{\Omega^\infty}\!\!\!\!\exp\!\left\{\!ix\frac{M_0(\varpi)}{L_0(\varpi)}\!\right\}
\!\left[\exp\!\left\{ix\!\sum_{k=1}^n\frac{M_k(\varpi)}{L_0(\varpi)\dots L_k(\varpi)}\!\right\}\right.\\
&&\left.\hspace{7ex}\widehat{g}\left(\frac{x}{L_0(\varpi)\dots L_n(\varpi)}\right)+
\widehat{g}\left(\frac{x}{L_0(\varpi)}\right)\right]P^\infty(d\varpi)\\
&=&\widehat{f}(0)\int_{\Omega^\infty}\exp\left\{ixZ_\infty(\varpi)\right\}P^\infty(d\varpi)+\sum_{n=1}^\infty I_{g,n}(x)\\
&=&\widehat{f}(x)-\widehat{g}(x)
\end{eqnarray*}
for every $x\in\mathbb R$. Finally, by Lemma \ref{lem1} we obtain that $f\in V_g$.
\end{proof}

Condition (\ref{m}) is satisfied in all important applications of inhomogeneous (and homogeneous) refinement equations. In the boundary case 
\begin{equation}\label{m1}
P(M=0)=1.
\end{equation}
we have the following result.

\begin{theorem}\label{thm1a}
Assume $f\in L^1(\mathbb R)$. If $(\ref{m1})$ holds and 
\begin{equation}\label{6a}
0<\int_\Omega\log|L(\omega)|P(d\omega)<+\infty,
\end{equation}
then $f\in V_g$ if and only if
\begin{equation}\label{7a}
\widehat{f}(x)=\widehat{f}(0)+\sum_{n=1}^\infty I_{g,n}(x)+\widehat{g}(x)
\end{equation}
for every $x\in\mathbb R$.
\end{theorem}

\begin{proof}
From (\ref{m1}) we conclude that series (\ref{4}) is well defined and $Z_\infty(\varpi)=0$ for almost all $\varpi\in\Omega^\infty$.  By the Kolmogorov strong law of large numbers we have
$$
\lim_{n\to\infty}\left(\prod_{k=1}^n|L_k(\varpi)|\right)^\frac{1}{n}=\exp\left\{\lim_{n\to\infty}\frac{1}{n}\sum_{k=1}^n\log|L_k(\varpi)|\right\}=\exp\left\{\int_\Omega\log|L(\omega)|P(d\omega)\right\}
$$
for almost all $\varpi\in\Omega^\infty$. Thus applying (\ref{6a}) we infer that (\ref{7b}) holds for all $x\in\mathbb R$ and almost all $\varpi\in\Omega^\infty$.

Assume $f\in V_g$. Then by the Lebesgue dominated convergence theorem we obtain
$\lim_{N\to\infty}I_{f,N}(x)=\widehat{f}(0)$ for every $x\in\mathbb R$. This jointly with Lemma \ref{lem2} shows that the series $\sum_{n=1}^\infty I_{g,n}(x)$ converges for every $x\in\mathbb R$ and, in consequence, (\ref{7a}) holds for every $x\in\mathbb R$.

The proof of the converse implication is similar to that of Theorem \ref{thm1}, with an evident  modification.
\end{proof}

Note that under conditions (\ref{m1}) and (\ref{6a}) the set $V_g$ consist of at most one member; indeed, if $f_1,f_2\in V_g$, then $f_1-f_2\in L^1(\mathbb R)$ and by Theorem \ref{thm1a} we have 
$$
\widehat{(f_1-f_2)}(x)=\widehat{f_1}(x)-\widehat{f_2}(x)=\widehat{f_1}(0)-\widehat{f_2}(0)
$$
for every $x\in\mathbb R$, hence $\widehat{f_1}(0)=\widehat{f_2}(0)$, and in consequence $f_1=f_2$.

If (\ref{6a}) holds, then the series (\ref{4}) converges absolutely for almost all $\varpi\in\Omega^\infty$ provided $\int_\Omega\log\max\{|L^{-1}(\omega)M(\omega)|,1\}P(d\omega)<+\infty$ and
\begin{equation}\label{cc}
P(cL+M=c)<1\hspace{3ex}\hbox{ for every }c\in\mathbb R
\end{equation}
(see \cite{G1974}, cf. \cite{GM2000}). Observe that (\ref{cc}) rules out (\ref{m1}); indeed putting $c=0$ in (\ref{cc}) we obtain (\ref{m}).  It is known that replacing condition (\ref{6a}) by 
\begin{equation}\label{6}
-\infty<\int_\Omega\log|L(\omega)|P(d\omega)<0,
\end{equation}
series (\ref{4})  diverges if (\ref{cc}) holds (see \cite{GM2000}). However, in the next result, convergence of series (\ref{4}) is not relevant.

\begin{theorem}\label{thm1b}
Assume $f\in L^1(\mathbb R)$.  If $(\ref{6})$ holds, then $f\in V_g$ if and only if
\begin{equation}\label{7}
\widehat{f}(x)=\sum_{n=1}^\infty I_{g,n}(x)+\widehat{g}(x)
\end{equation}
for every $x\in\mathbb R$.
\end{theorem}

\begin{proof}
Assume $f\in V_g$. Condition (\ref{6}) jointly with the Kolmogorov strong law of large numbers  implies
$$
\lim_{N\to\infty}\left|\frac{x}{L_1(\varpi)\cdots L_N(\varpi)}\right|=+\infty
$$
for all $x\in\mathbb R\setminus\{0\}$ and almost all $\varpi\in\Omega^\infty$. This jointly with the Lebesgue dominated convergence theorem and the Lebesgue-Riemann lemma yields
$$
\lim_{N\to\infty}|I_{f,N}(x)|\leq\int_{\Omega^\infty}\lim_{N\to\infty}\left|\widehat{f}\left(\frac{x}{L_1(\varpi)\cdots L_N(\varpi)}\right)\right|P^\infty(d\varpi)=0
$$
for every $x\in\mathbb R\setminus\{0\}$.  Therefore, Lemma \ref{lem2} shows that the series $\sum_{n=1}^\infty I_{g,n}(x)$ converges for every $x\in\mathbb R\setminus\{0\}$ and, in consequence, (\ref{7}) holds for every $x\in\mathbb R\setminus\{0\}$. By the continuity of  $\widehat{f}$ and $\widehat{g}$ we infer that (\ref{7}) also holds for $x=0$.

The proof of the converse implication is similar to that of Theorem \ref{thm1}, with an evident  modification.
\end{proof}

Note that under condition (\ref{6}) the set $V_g$ consists of at most one member, and moreover, if $f\in V_g$, then $\widehat{f}(0)=0$.
 
We end this section pointing out that in the critical case 
\begin{equation}\label{6b}
\int_\Omega \log |L(\omega)|P(d\omega)=0
\end{equation}
series (\ref{4}) diverges in the case where (\ref{cc}) holds (see \cite{BBE1997}, cf. \cite{GM2000}) or converges in the case where (\ref{m1}) holds. But, independent of the case, it is well known that under assumption (\ref{6b}) equation (\ref{e}) has completely different behavior than under assumption (\ref{6a}) or (\ref{6}). The next two examples show that in the critical case equation (\ref{e}) can have lots of integrable solutions.

\begin{example}\label{ex1}
Put  $P(\omega)=\frac{1}{2}$,  $L(\omega)=\omega$ and $M(\omega)=0$ for every $\omega\in\Omega=\{-1,1\}$. Then $(\ref{6b})$ and $(\ref{m1})$ hold, and moreover, equation $(\ref{e})$ takes the form
\begin{equation}\label{e1}
f(x)=\frac{1}{2}f(-x)+\frac{1}{2}f(x)+g(x).
\end{equation}
It is easy to see, taking $g$ to be even and $h\in L^1(\mathbb R)$ to be odd, that the function $f=h+g$ satisfies $(\ref{e1})$.
\end{example}

\begin{example}\label{ex2}
Put  $P(\omega)=\frac{1}{2}$, $L(\omega)=\omega$ and $M(\omega)=1-\omega$ for every $\omega\in\Omega=\{-1,1\}$. Then $(\ref{6b})$ and $(\ref{cc})$ hold, and moreover, equation $(\ref{e})$ takes the form
\begin{equation}\label{e2}
f(x)=\frac{1}{2}f(-x+2)+\frac{1}{2}f(x)+g(x).
\end{equation}
It is easy to see, taking $g$ to have  the graph symmetric with respect to the point $(1,0)$ and $h\in L^1(\mathbb R)$ to have the graph symmetric with respect to the vertical line $x=1$, that the function $f=h+g$ satisfies $(\ref{e2})$.
\end{example}

The above examples show why we are not interested in the critical case in this paper. More details on equation (\ref{e}) in the critical case can be found in \cite{KCG1990}.

\section{Characterization of $V_g$ by Radon-Nikodym derivatives} 

Throughout this section we assume that $L$ is positive almost everywhere on $\Omega$; the case where $L$ is negative almost everywhere on $\Omega$ is also acceptable for our purpose, but it will not be considered in this paper. 

It is known that every function from the set $V_0$ is determined, up to a  multiplicative constant, by the Radon-Nikodym derivative of their integrals over $(-\infty,x]$, $x\in\mathbb R$ with respect to the one dimensional Lebesgue measure (see \cite[Section 3.4]{KM2013}). It turns out that such a property also have all elements from the set $V_g$. To be more precise let us define a function $G\colon\mathbb R\to\mathbb R$ by
$$
G(x)=\int_{-\infty}^x g(t)dt.
$$
It is clear that the function $G$ is bounded.

We have the following counterpart of Proposition 2 from \cite{KM2008a}.

\begin{proposition}\label{prop1}
Assume $f\in L^1(\mathbb R)$. Then $f\in V_g$ if and only if the function $F\colon\mathbb R\to\mathbb R$ given by 
\begin{equation}\label{8}
F(x)=\int_{-\infty}^xf(t)dt
\end{equation} 
satisfies 
\begin{equation}\label{d}
F(x)=\int_{\Omega}F(L(\omega)x-M(\omega))P(d\omega)+G(x)
\end{equation}
for every $x\in\mathbb R$.
\end{proposition}

\begin{proof}
Assume $f\in V_g$. By the Fubini theorem we obtain
\begin{eqnarray*}
F(x)&=&\int_{-\infty}^xf(t)dt=
\int_{\Omega}\int_{-\infty}^xL(\omega)f(L(\omega)t-M(\omega))dtP(d\omega)+\int_{-\infty}^xg(t)dt\\
&=&\int_{\Omega}\int_{-\infty}^{L(\omega)x-M(\omega)}f(y)dyP(d\omega)+G(x)=\int_{\Omega}F(L(\omega)x-M(\omega))P(d\omega)+G(x)
\end{eqnarray*}
for every $x\in\mathbb R$.

Conversely, if a function $F\colon\mathbb R\to\mathbb R$ given by (\ref{8}) satisfies (\ref{d}) for every $x\in\mathbb R$, then Fubini theorem yields
$$
\int_{-\infty}^xf(t)dt=F(x)=
\int_{-\infty}^x\left[\int_{\Omega}L(\omega)f(L(\omega)t-M(\omega))P(d\omega)+g(t)\right]dt
$$
for every $x\in\mathbb R$. Hence $f\in V_g$.
\end{proof}

Now we need to write a precise formula for iterates (in the sense of Kuczma and Baron \cite{BK1977}) of the random map $\varphi\colon\mathbb R\times\Omega\to\mathbb R$ given by (\ref{rap}); the iterates are defined as follows:
$$
\varphi^1(x,\varpi)=\varphi(x,\omega_1)\hspace{3ex}\hbox{ and }\hspace{3ex}
\varphi^{n+1}(x,\varpi)=\varphi(\varphi^n(x,\varpi),\omega_{n+1})
$$
for all $x\in\mathbb R$ and $\varpi=(\omega_1,\omega_2,\dots)\in\Omega^{\infty}$.

An easy calculation shows that
$$
\varphi^n(x,\varpi)=x\prod_{k=1}^n L_k(\varpi)-\sum_{i=1}^n M_i(\varpi)\prod_{j=i+1}^n L_j(\varpi)
$$
for all $n\in\mathbb N$, $x\in\mathbb R$ and $\varpi\in\Omega^{\infty}$. Then, defining a 
random variable $Z_n\colon\Omega^\infty\to\mathbb R$ by
\begin{equation}\label{Z}
Z_n(\varpi)=-\sum_{i=1}^n M_i(\varpi)\prod_{j=1}^{n-i} L_j(\varpi)
\end{equation}
and a  transformation  $\sigma_n\colon\Omega^\infty\to\Omega^\infty$ by $\sigma_n(\omega_1,\omega_2,\dots)=(\omega_n,\dots,\omega_1,\omega_{n+1},\dots)$ we obtain
$$
\varphi^n(x,\sigma_n(\varpi))=x\prod_{k=1}^n L_k(\varpi)+Z_n(\varpi)
$$
for all  $n\in\mathbb N$, $x\in\mathbb R$ and $\varpi\in\Omega^{\infty}$.  It can be inferred from \cite{G1974} (see also \cite{GM2000}) that assuming (\ref{6}) and
\begin{equation}\label{10}
\int_\Omega\log\max\{|M(\omega)|,1\}P(d\omega)<\infty 
\end{equation} 
the sequence of iterates $(\varphi^n(x,\sigma_n(\varpi)))_{n\in\mathbb N}$ converges for all $x\in\mathbb R$ and almost all $\varpi\in\Omega^\infty$,  and moreover, the limit is independent of $x\in\mathbb R$. Since for all $x\in\mathbb R$ and $n\in\mathbb N$ random variables $\varphi^n(x,\cdot)$ and $\varphi^n(x,\sigma_n(\cdot))$ have the same distributions, it follows that for every $x\in\mathbb R$ the sequence $(\varphi^n(x,\cdot))_{n\in\mathbb N}$ converges in distribution to a random variable $\phi$, which is independent of $x\in\mathbb R$. In particular, the sequence $(Z_n)_{n\in\mathbb N}$ defined by (\ref{Z}) converges in distribution to a random variable $\phi$. Denoting by $\Phi$ the probability distribution function of $\phi$ we can formulate the following result.

\begin{theorem}\label{thm2}
Assume that $M$  is integrable and let
\begin{equation}\label{6c}
\int_\Omega L(\omega)P(d\omega)<1.
\end{equation}
If 
\begin{equation}\label{11}
\int_{\mathbb R}g(t)\Phi(t)dt=\widehat{g}(0)
\end{equation}
and $G$ is Lipschitzian, then equation $(\ref{d})$ has a Lipschitzian solution $F\colon\mathbb R\to\mathbb R$. 

Moreover, if there exists $f\in L^1(\mathbb R)$ such that $(\ref{8})$ holds for every $x\in\mathbb R$, then $f\in V_g$.
\end{theorem}

The proof of Theorem \ref{thm2} is based on assertion (iii) of Corollary 4.1 from \cite{B2009};  this is the real reason for our assumptions about the random variable $L$. For the convenience of the reader, we
repeat this assertion as a lemma formulated for requirements of our needs.

\begin{lemma}\label{lemB}
Assume that $M$  is integrable and let $(\ref{6c})$ holds. If $G$ is Lipschitzian and there exists $x_0\in X$ such that
$$
\lim_{n\to\infty}\int_{\Omega^\infty}G(\varphi^n(x_0,\varpi))P^\infty(d\varpi)=0,
$$
then equation $(\ref{d})$ has a Lipschitzian solution $F\colon\mathbb R\to\mathbb R$. 
\end{lemma}

\noindent{\it Proof of Theorem \ref{thm2}.} 
We first note that condition (\ref{6c}) implies condition (\ref{6}) and integrability of $M$ forces condition (\ref{10}).

For every $n\in\mathbb N$ denote by $\Phi_n$ the probability distribution function of the random variable $Z_n$ defined by (\ref{Z}). Then for every $t\in\mathbb R$, being a point of continuity of $\Phi$, we have $\lim_{n\to\infty}\Phi_n(t)=\Phi(t)$. Hence $\lim_{n\to\infty}g(t)\Phi_n(t)=g(t)\Phi(t)$ for almost all $t\in\mathbb R$. In consequence,  by the Lebesgue dominated convergence theorem and (\ref{11}) we obtain
$$
\lim_{n\to\infty}\int_{\mathbb R}g(t)\Phi_n(t)dt=\int_{\mathbb R}g(t)\Phi(t)dt=\widehat{g}(0).
$$
This jointly with the Fubini theorem gives
\begin{eqnarray*}
\lim_{n\to\infty}\int_{\Omega^\infty}G(\varphi^n(0,\varpi))P^\infty(d\varpi)
&=&\lim_{n\to\infty}\int_{\Omega^\infty}G(\varphi^n(0,\sigma_n(\varpi)))P^\infty(d\varpi)\\ &=&\lim_{n\to\infty}\int_{\Omega^\infty}G(Z_n(\varpi))P^\infty(d\varpi)\\ 
&=&\lim_{n\to\infty}\int_{\Omega^\infty}\int_{-\infty}^{Z_n(\varpi)}g(t)dtP^\infty(d\varpi)\\
&=&\lim_{n\to\infty}\int_{\{(t,\varpi)\in\mathbb R\times\Omega^\infty:t<Z_n(\varpi)\}}g(t)dt P^\infty(d\varpi)\\
&=&\lim_{n\to\infty}\int_{\mathbb R}\int_{\{\varpi\in\Omega^\infty:t<Z_n(\varpi)\}}g(t)P^\infty(d\varpi)dt\\
&=&\lim_{n\to\infty}\int_{\mathbb R}g(t)P^\infty(\{\varpi\in\Omega^\infty:t<Z_n(\varpi)\})dt\\
&=&\lim_{n\to\infty}\int_{\mathbb R}g(t)[1-P^\infty(\{\varpi\in\Omega^\infty:Z_n(\varpi)\leq t\})dt\\
&=&\widehat{g}(0)-\lim_{n\to\infty}\int_{\mathbb R}g(t)\Phi_n(t)dt=0.
\end{eqnarray*}

Now the main part of the statement follows from Lemma \ref{lemB}.

The moreover statement follows from Proposition \ref{prop1}.
\hfill $\square$ 
\medskip 

It is known that each absolutely continuous function $F\colon{\mathbb R}\to\mathbb R$ of bounded variation with $\lim_{x\to-\infty}F(x)=0$ has a suitably understood Radon-Nikodym derivative (due to the Hahn decomposition), i.e., $F$ is represented by (\ref{8}) with $f\in L^1(\mathbb R)$. However,  the following example shows that not all absolutely continuous real functions are represented in such a way.

\begin{example}
There exists a bounded and Lipschitzian function $F\colon\mathbb R\to\mathbb R$ for which does not exist a function $f\in L^1(\mathbb R)$ such that $(\ref{8})$ holds for every $x\in\mathbb R$.
\end{example}

\begin{proof}
Define a function $F\colon\mathbb R\to\mathbb R$ by the Riemann improper integrals
$$
F(x)=\int_{-\infty}^xh(t)dt,
$$
where $h\colon\mathbb R\to\mathbb R$ is given by
$$
h(x)=\left\{ \begin{array}{ccl}
\frac{1}{n+1} & \hbox{ for } & x\in \bigcup_{n=0}^\infty (-2n-1,-2n],\\
-\frac{1}{n+1} & \hbox{ for } & x\in \bigcup_{n=0}^\infty (-2n-2,-2n-1],\\
0 & \hbox{ for } & x\in(0,+\infty).
\end{array}\right.
$$

It is easy to show that $F$ is Lipschitzian and bounded. 
Suppose, to derive a contradiction, that there exists $f\in L^1(\mathbb R)$ such that (\ref{8}) holds for every $x\in X$. Since $\int_x^y f(t)dt=\int_x^y h(t) dt$ for all $x,y\in\mathbb R$, we have $h=f$ 
almost everywhere on $\mathbb R$. Hence $h\in L^1(\mathbb R)$, a contradiction.
\end{proof}

\section*{Acknowledgements}
This research was supported by University of Silesia Mathematics Department (Iterative Functional Equations and Real Analysis program).

The authors are grateful to the referee for a number of helpful suggestions for improvement in the article.

\end{document}